\documentclass[a4paper]{article}

\usepackage{float,ifthen}
\newboolean{annotate}

\setboolean{annotate}{false}
\usepackage[color]{changebar}
\cbcolor{red}

\newcommand{\annotate}[1]{\ifthenelse{\boolean{annotate}}{\cbstart \textcolor{red}{#1} \cbend}{#1}}

\usepackage{graphicx}
\usepackage{amsmath,amssymb,amsfonts,amsthm}
\usepackage{hyperref}
\usepackage{rotating}
\usepackage{array}
\usepackage{booktabs}
\usepackage{calc}
\usepackage{enumitem}
\usepackage{etoolbox}
\usepackage[margin=2.5cm]{geometry}
\usepackage{hyperref}
\usepackage{parskip}
\usepackage{pdflscape}
\usepackage{qtree}
\usepackage{setspace}
\usepackage{subfig}
\usepackage{tabu}
\usepackage{tcolorbox}
\usepackage{tikz-qtree}
\usetikzlibrary{arrows,calc,decorations.pathreplacing}
\tikzset{>=stealth'}
\usepackage{xcolor}
\usepackage[section]{placeins}
\usepackage{tabto}

\newtcbox{\cbox}[1]{
	boxrule=1pt,
	colback=#1!10!white,
	colframe=#1,
	left=0pt,right=0pt,top=0pt,bottom=0pt
}

\newcommand{\raisedcbox}[2]{
	\raisebox{-4pt}[10pt][8pt]{\cbox{#1}{#2}}
}

\newcommand{\snp}{\raisedcbox{red}{sNP}}
\newcommand{\np}{\raisedcbox{yellow}{NP}}

\newcommand{\p}{\raisedcbox{green}{P}}

\newcommand{\multiline}[2]{
	\begin{tabular}{@{}#1@{}}#2\end{tabular}
}

\newcommand{\Ain}[1]{A_{#1}^{\mathrm{in}}}
\newcommand{\Aout}[1]{A_{#1}^{\mathrm{out}}}

\theoremstyle{plain}
\newtheorem{theorem}{Theorem}
\newtheorem{lemma}{Lemma}

\theoremstyle{definition}
\newtheorem{remark}{Remark}

\newcommand{\reals}{\mathbb{R}}
\newcommand{\vect}[1]{\mbox{\boldmath$#1$}}
\DeclareMathOperator{\val}{val}
\DeclareMathOperator{\LP}{LP}

\title{A polynomially solvable case of the pooling problem}

\author
{
Natashia Boland \\
{\it Georgia Institute of Technology, Atlanta, U.S.A.} \and
Thomas Kalinowski \qquad Fabian Rigterink \\
{\it The University of Newcastle, Australia}
}

\date{\today}

\begin{document}

\onehalfspacing

\maketitle

\begin{abstract}
  \noindent Answering a question of Haugland, we show that the pooling problem with one pool and a
  bounded number of inputs can be solved in polynomial time by solving a polynomial number of linear
  programs of polynomial size. We also give an overview of known complexity results and remaining
  open problems to further characterize the border between (strongly) NP-hard and polynomially
  solvable cases of the pooling problem.

\paragraph{Keywords} Pooling problem $\cdot$ Computational complexity
\end{abstract}

\section{Introduction, motivation and problem definition}

The pooling problem is a nonconvex nonlinear programming problem with applications in the refining
and petrochemical industries \cite{deWitt89,Rigby95}, mining \cite{Boland15a,Boland15c},
agriculture, food manufacturing, and pulp and paper production \cite{Visweswaran09}. Informally, the
problem can be stated as follows: given a set of raw material suppliers (inputs) and qualities of
the material, find a cost-minimizing way of blending these raw materials in intermediate pools and
outputs so as to satisfy requirements on the final output qualities. The blending in pools and
outputs introduces bilinear constraints and makes the problem hard.

While the pooling problem has been known to be hard in practice ever since its proposal by Haverly
in 1978 \cite{Haverly78}, it was only formally proven to be strongly NP-hard by Alfaki and Haugland
in 2013 \cite{Alfaki13a}. Their proof of strong NP-hardness, however, considered a very general case
of the problem, with arbitrary parameters and an arbitrary network structure. Once the parameters
and the network structure are more specific, e.g., by bounding the number of vertices, their in- and
out-degrees, or the number of qualities, the complexity of the problem needs to be re-examined. This
way, several polynomially solvable cases of the pooling problem were proven
\cite{Alfaki13b,Haugland14,Haugland15}. However, the border between (strongly) NP-hard and
polynomially solvable cases of the pooling problem is still only partially characterized. This is
mainly due to the combinatorial explosion of parameter choices for the problem. In this paper, we
solve an open problem that has been pointed out in \cite{Haugland14,Haugland15}: the pooling problem
with one pool and a bounded number of inputs is in fact polynomially solvable.

\begin{table}
\caption{Notation for the pooling problem}
\label{tab:notation}
\setlength{\tabcolsep}{4pt}
\subfloat{%
\begin{tabu} to .4\linewidth {lX[l]}
\multicolumn{2}{l}{Sets} \\
\midrule
$V$ & Set of vertices \\
$I$ & Set of inputs \\
$L$ & Set of pools \\
$J$ & Set of outputs \\
$K$ & Set of qualities \\
$A$ & Set of arcs \\
$A_I$ & Set of input-to-pool arcs: \\
& \quad $A_I := A \cap (I \times L)$ \\
$A_J$ & Set of pool-to-output arcs: \\
& \quad $A_J := A \cap (L \times J)$ \\
$\Aout{v}$ & Set of outgoing arcs of $v \in I \cup L$ \\
$\Ain{v}$ & Set of incoming arcs of $v \in L \cup J$ \\
\end{tabu}
}
\hfill
\subfloat{%
\begin{tabu} to .56\linewidth {lX[l]}
\multicolumn{2}{l}{Parameters} \\
\midrule
$c_a$ & Cost of flow on arc $a \in A$ \\
$\lambda_{ik}$ & Quality value of input $i \in I$ for quality $k \in K$ \\
$\lambda_{ak}$ & $\lambda_{ak} \equiv \lambda_{ik}, \enspace a \in \Aout{i}, \enspace i \in I, \enspace k \in K$ \\
$\mu_{jk}$ & Upper bound on quality value of output $j \in J$ for quality $k \in K$ \\
$C_v$ & Upper bound on total flow through vertex $v \in V$ \\
$u_a$ & Upper bound on flow on arc $a \in A$ \\
& \\
\multicolumn{2}{l}{Variables} \\
\midrule
$x_a$ & Flow on arc $a \in A_I$ \\
$y_a$ & Flow on arc $a \in A_J$ \\
$p_{\ell k}$ & Quality value of pool $\ell \in L$ for quality $k \in K$ \\
$p_{ak}$ & $p_{ak} \equiv p_{\ell k}, \enspace a \in \Aout{\ell}, \enspace \ell \in L, \enspace k \in K$ \\
\end{tabu}
}
\end{table}

We consider a directed graph $G = (V, A)$ where $V$ is the set of vertices and $A$ is the set of
arcs. $V$ is partitioned into three subsets $I, L, J \subset V$: $I$ is the set of inputs, $L$ is
the set of pools and $J$ is the set of outputs. Flows are blended in pools and outputs. The pooling
problem literature addresses a variety of problem instances with $A \subseteq (I \times L) \cup (L
\times L) \cup (L \times J) \cup (I \times J)$. Instances with $A \cap (L \times L) = \emptyset$ are
referred to as \emph{standard pooling problems} (SPPs), and instances with $A \cap (L \times L) \neq
\emptyset$ are referred to as \emph{generalized pooling problems} (GPPs). Both SPPs and GPPs can be
modelled as bilinear programs, which are special cases of nonlinear programs. Instances with $L =
\emptyset$ are referred to as \emph{blending problems} and can be modelled as linear programs.

In this paper (as in \cite{Alfaki13b,Haugland14,Haugland15}), we study the complexity of SPPs where
$A \subseteq (I \times L) \cup (L \times J)$, i.e., all arcs are either input-to-pool or
pool-to-output arcs. For notational simplicity, we denote the set of the former by $A_I := A \cap (I
\times L)$ and the set of the latter by $A_J := A \cap (L \times J)$. We do not consider
input-to-output arcs since for every such arc $(i, j)$, we can add an auxiliary pool $\ell$ and
replace $(i, j)$ by an input-to-pool arc $(i, \ell)$ and a pool-to-output arc $(\ell,
j)$. Throughout this paper, we use the term \emph{pooling problem} to refer to a SPP without
input-to-output arcs. We consider a set of qualities $K$ whose quality values are tracked across the
network. We assume linear blending, i.e., the quality value of a pool or output for a quality is the
convex combination of the incoming quality values weighted by the incoming flows as a
fraction of the total incoming flow.

For inputs and pools $v \in I \cup L$, we denote the set of outgoing arcs of $v$ by $\Aout{v}$, and
for pools and outputs $v \in L \cup J$, we denote the set of incoming arcs of $v$ by $\Ain{v}$. Let
$x_a$ be the flow on input-to-pool arc $a \in A_I$, and let $y_a$ be the flow on pool-to-output arc
$a \in A_J$. The cost of flow on arc $a \in A$ (which may be negative) is given by $c_a$. The total
flow through vertex $v \in V$ (resp. the flow on arc $a \in A$) is bounded above by $C_v$
(resp. $u_a$). For every input $i \in I$ and quality $k \in K$, the quality value of the incoming
raw material is given by $\lambda_{ik}$. Let $p_{\ell k}$ denote the quality value of the blended
raw materials in pool $\ell \in L$ for quality $k \in K$. For every output $j \in J$ and quality $k
\in K$, the upper bound on the quality value of the outgoing blend is given by $\mu_{jk}$. In
addition to $\lambda_{ik}$ and $p_{\ell k}$, it is sometimes more convenient to have arc-based
rather than node based quality parameters and variables. Since the quality of flow on arc $(v, w)$
is equal to the blended quality of the total flow through vertex $v$, we have $\lambda_{ik} \equiv
\lambda_{ak}$ for all inputs $i \in I$, their outgoing arcs $a \in \Aout{i}$ and qualities $k \in
K$. Analogously, we have $p_{\ell k} \equiv p_{ak}$ for all pools $\ell \in L$, their outgoing arcs
$a \in \Aout{\ell}$ and qualities $k \in K$. Table~\ref{tab:notation} summarises the notation for
the pooling problem.

\begin{comment}
\begin{table}
\caption{Notation for the pooling problem}
\label{tab:notation}
\begin{tabu}{lX}
\multicolumn{2}{l}{Sets} \\
\midrule
$V$ & Set of vertices \\
$I$ & Set of inputs \\
$L$ & Set of pools \\
$J$ & Set of outputs \\
$K$ & Set of qualities \\
$A$ & Set of arcs \\
$A_I$ & Set of input-to-pool arcs: $A_I := A \cap (I \times L)$ \\
$A_J$ & Set of pool-to-output arcs: $A_J := A \cap (L \times J)$ \\
$\Ain{v}$ & Set of incoming arcs of $v \in L \cup J$ \\
$\Aout{v}$ & Set of outgoing arcs of $v \in I \cup L$ \\
& \\
\multicolumn{2}{l}{Parameters} \\
\midrule
$c_a$ & Cost of flow on arc $a \in A$ \\
$\lambda_{ik}$ & Quality value of input $i \in I$ for quality $k \in K$ \\
$\lambda_{ak}$ & $\lambda_{ak} \equiv \lambda_{ik}, \enspace a \in \Aout{i}, \enspace i \in I, \enspace k \in K$ \\
$\mu_{jk}$ & Upper bound on quality value of output $j \in J$ for quality $k \in K$ \\
$C_v$ & Upper bound on total flow through $v \in V$ \\
$u_a$ & Upper bound on flow $y_a, \, a \in A$ \\
& \\
\multicolumn{2}{l}{Variables} \\
\midrule
$x_a$ & Flow on arc $a \in A_I$ \\
$y_a$ & Flow on arc $a \in A_J$ \\
$p_{\ell k}$ & Quality value of pool $\ell \in L$ for quality $k \in K$ \\
$p_{ak}$ & $p_{ak} \equiv p_{\ell k}, \enspace a \in \Aout{\ell}, \enspace \ell \in L, \enspace k \in K$ \\
\end{tabu}
\end{table}
\end{comment}

We now present the classical formulation of the pooling problem, commonly referred to as the
P-formulation \cite{Haverly78}. There are numerous alternative formulations of the pooling problem,
including the Q-~\cite{BenTal94}, PQ-~\cite{Tawarmalani02} and HYB-formulations~\cite{Audet04}, and
most recently multi-commodity flow formulations~\cite{Alfaki13a,Alfaki13b,Boland15b}. All
formulations are equivalent in the sense that there is a one-to-one correspondence between a
feasible solution of one formulation and another, and they all have the same optimal objective
value. However, the alternative formulations often show a better computational performance than the
P-formulation, as studied e.g. in \cite{Boland15b}. A recent paper by Gupte et al. \cite{Gupte15}
gives an excellent overview of topics that have been studied in the context of the pooling
problem. Within the scope of this paper, however, we chose to prove complexity results using the
classical P-formulation.

In the P-formulation, a flow $(\vect{x}, \vect{y})$ satisfies the following constraints:
\begin{alignat}{2}
\sum \limits_{a \in \Ain{\ell}} x_a & = \sum \limits_{a \in \Aout{\ell}} y_a, \qquad & & \ell \in L, \label{eq:1} \\
\sum \limits_{a \in \Aout{i}} x_a & \leqslant C_i, & & i \in I, \label{eq:2} \\
\sum \limits_{a \in \Ain{\ell}} x_a & \leqslant C_{\ell}, & & \ell \in L, \label{eq:3} \\
\sum \limits_{a \in \Ain{j}} y_a & \leqslant C_j, & & j \in J, \label{eq:4} \\
x_a, y_a & \leqslant u_a, & & a \in A_I, \, A_J, \, \textnormal{resp.} \label{eq:5}
\end{alignat}
Constraint~\eqref{eq:1} is flow conservation which ensures that at every pool, the total incoming
flow equals the total outgoing flow. \eqref{eq:2}--\eqref{eq:4} are vertex capacity constraints and
\eqref{eq:5} is an arc capacity constraint. For notational simplicity, we denote the \emph{set of
  flows} by $\mathcal{F} := \{ (\vect{x}, \vect{y}) \in \mathbb{R}_{\geqslant 0}^{|A_I|} \times
\mathbb{R}_{\geqslant 0}^{|A_J|}: \, \textnormal{\eqref{eq:1}--\eqref{eq:5} are satisfied} \}$. The
P-formulation can now be stated as follows:
\begin{alignat}{3}
\underset{\vect{x}, \vect{y}, \vect{p}}{\min} \qquad & & \omit\rlap{$\displaystyle \sum \limits_{a \in A_I} c_a x_a + \sum \limits_{a \in A_J} c_a y_a$} \notag \\
\textnormal{s.t.} \qquad & & (\vect{x}, \vect{y}) & \in \mathcal{F}, & & \notag \\
& & \sum \limits_{a \in \Ain{\ell}} \lambda_{ak} x_a & = p_{\ell k} \sum \limits_{a \in \Aout{\ell}} y_a, \qquad & & \ell \in L, \enspace k \in K, \label{eq:6} \\
& & \sum \limits_{a \in \Ain{j}} p_{ak} y_a & \leqslant \mu_{jk} \sum \limits_{a \in \Ain{j}} y_a, & & j \in J, \enspace k \in K. \label{eq:7}
\end{alignat}
Equality \eqref{eq:6} is the pool blending constraint which ensures that the $p$ variables track the
quality values across the network. Inequality \eqref{eq:7} is the output blending constraint. We
take the requirements that $\lambda_{ak} \equiv \lambda_{ik}$ for all $a \in \Aout{i}$, $i \in I$ and $k \in K$, and that $p_{ak} \equiv p_{\ell k}$ for all $a \in \Aout{\ell}$, $\ell \in L$ and
$k \in K$, to be implicit in the model.

\section{Known complexity results}

Table~\ref{tab:complexityResultsTable} provides an overview of known complexity results, and
Figure~\ref{fig:complexityResultsFigure} shows most of these complexity results in a tree
structure. All of these results were formally proven in
  \cite{Alfaki13b,Dey15,Haugland14,Haugland15}. When bounding the number of vertices, the cases of
one input or output are polynomially solvable. Furthermore, the cases of one pool and a bounded
number of outputs or qualities are polynomially solvable. If we only have one pool (and no other
restrictions), then the problem remains strongly NP-hard. The same holds if we have only one
quality. The problem remains strongly NP-hard if we have one quality and two inputs or two
outputs. Only if we have one quality, two inputs and two outputs, then the problem becomes
NP-hard. The problem also remains strongly NP-hard if the out-degrees of inputs and pools
are bounded above by two, or if the in-degrees of pools and outputs are bounded above by
two. Finally, it was shown in \cite{Dey15} that there exists a polynomial time algorithm
  which guarantees an $n$-approximation (where $n$ is the number of output nodes). The authors of
  this paper also showed that if there exists a polynomial time approximation algorithm with
  guarantee better than $n^{1 - \varepsilon}$ for any $\varepsilon > 0$, then NP-complete problems
  have randomized polynomial time algorithms.

% Thorem 1: Let $n$ be the number of output nodes in a pooling problem. There exists a polynomial time algorithm for the pooling problem that guarantees an $n$-approximation. On the other hand, if there exists a polynomial time approximation algorithm with guarantee better than $n^{1 - \varepsilon}$ for any $\varepsilon > 0$ for the pooling problem, then NP-complete problems have randomized polynomial-time algorithms.

%\begin{landscape}
\begin{sidewaystable}[htb]
\centering
\caption{Overview of known complexity results}
\label{tab:complexityResultsTable}
\medskip
\small
\begin{tabular}{r|ccc|c|ccc|ccl}
\toprule
& \multicolumn{3}{c|}{bounded \#vertices} & & \multicolumn{3}{c|}{bounded in-\slash out-degrees} & & & \\
\# & $|I|$ & $|L|$ & $|J|$ & $|K|$ & $\forall \, i \in I$ & $\forall \, \ell \in L$ & $\forall \, j \in J$ & Complexity & Reduction & Reference(s) \\
\midrule
1 & 1 & & & & & & & \p & & trivial \\ % \cite{Haugland14b}, 40\slash 65
2 & & 1 & & & & & & \snp & MIS & \multiline{l}{\cite{Alfaki13b}, Corollary~1; \\ \cite{Haugland14}, Proposition~1; \\ \cite{Haugland15}, Theorems~1--2} \\ % \cite{Haugland14b}, 19\slash 65
3 & & & 1 & & & & & \p & & trivial \\ % \cite{Haugland14b}, 40\slash 65
4 & & & & 1 & & & & \snp & X3C & see \#8, \#9 and \#11 \\ % \cite{Haugland14b}, 24\slash 65
\midrule
5 & $[1, i_{\max}]$ & 1 & & & & & & \p & & this paper \\
6 & & 1 & $[1, j_{\max}]$ & & & & & \p & & \cite{Haugland14}, Proposition~2 \\
7 & & 1 & & $[1, k_{\max}]$ & & & & \p & & \multiline{l}{\cite{Alfaki13b}, Proposition~2; \\ \cite{Haugland14}, Proposition~3} \\
\midrule
8 & 2 & & & 1 & & & & \snp & X3C & \cite{Haugland15}, Theorem~4 \\
% 9 & & 2 & & 1 & & & & \np & BP2 & \cite{Haugland14}, Proposition~5 \\
9 & & & 2 & 1 & & & & \snp & X3C & \cite{Haugland15}, Theorem~5 \\
%10 & 2 & 2 & & 1 & & & & \np & BP2 & see \#11 \\ % \cite{Haugland14b}, 38\slash 65
10 & 2 & & 2 & 1 & & & & \np & BP2 & \multiline{l}{\cite{Haugland14}, Proposition~5; \\ \cite{Haugland15}, Theorem~3} \\ % \cite{Haugland14b}, 38\slash 65
%12 & & 2 & 2 & 1 & & & & \np & BP2 & see \#11 \\ % \cite{Haugland14b}, 38\slash 65
11 & \multicolumn{3}{c|}{$\min \{ |I|, |J| \} = 2$} & 1 & \multicolumn{3}{c|}{$\max \{ |\Ain{\ell}|, |\Aout{\ell}| \} \leqslant 6$} & \snp & X3C & \cite{Haugland15}, Corollary~1 \\
\midrule
12 & & & & & $|\Aout{i}| \leqslant 2$ & $|\Aout{\ell}| \leqslant 2$ & & \snp & MAX 2-SAT & \multiline{l}{\cite{Haugland14}, Proposition~7; \\ \cite{Haugland15}, Theorem~6} \\
13 & & & & & & $|\Ain{\ell}| \leqslant 2$ & $|\Ain{j}| \leqslant 2$ & \snp & MIN 2-SAT & \multiline{l}{\cite{Haugland14}, Proposition~6; \\ \cite{Haugland15}, Theorem~7} \\
14 & & & & & \multicolumn{3}{c|}{$\min \{ |\Ain{\ell}|, |\Aout{\ell}| \} = 1$} & \p & & \multiline{l}{\cite{Dey15}, Corollary~1; \\ \cite{Haugland14}, Proposition~4; \\ \cite{Haugland15}, Proposition~3} \\ % \cite{Haugland14b}, 41\slash 65
\bottomrule
\addlinespace[10pt]
\multicolumn{11}{c}{\p = polynomial, \np = NP-hard, \snp = strongly NP-hard,} \\
\addlinespace
\multicolumn{11}{c}{BP2 = bin packing with 2 bins, \enspace MAX 2-SAT = maximum 2-satisfiability, \enspace MIN 2-SAT = minimum 2-satisfiability,} \\
\multicolumn{11}{c}{MIS = maximal independent set, \enspace X3C = exact cover by 3-sets}
\end{tabular}
\end{sidewaystable}
%\end{landscape}

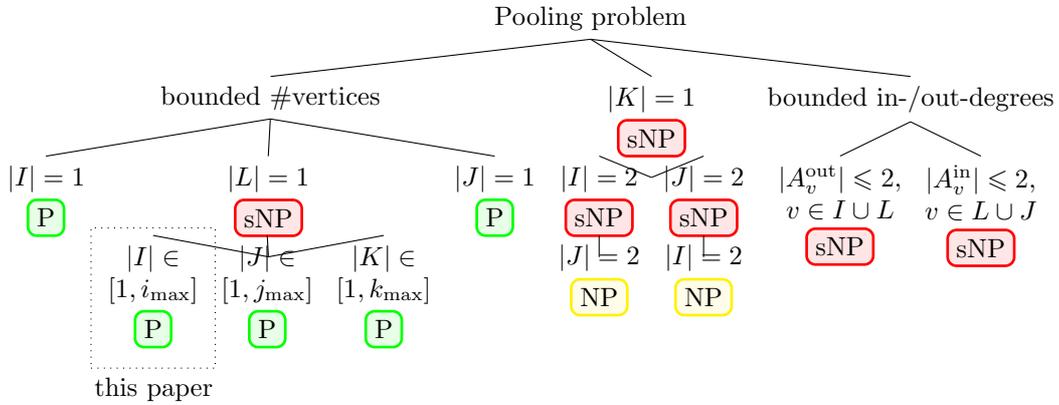
\begin{figure}[htb]
\centering
\begin{tikzpicture}
\tikzset{every tree node/.style={align=center,anchor=north}}
\tikzset{level 1/.style={level distance=40pt}}
\tikzset{level 2+/.style={level distance=60pt}}
\Tree
[.\node(r){Pooling problem};
	[.\node(n1){bounded \#vertices};
		\node(n11){$|I| = 1$ \\[2pt] \p};
		[.\node(n12){$|L| = 1$ \\[2pt] \snp};
			\node(n121){$|I| \in$ \\ $[1, i_{\max}]$ \\[2pt] \p};
			\node(n122){$|J| \in$ \\ $[1, j_{\max}]$ \\[2pt] \p};
			\node(n123){$|K| \in$ \\ $[1, k_{\max}]$ \\[2pt] \p};
		]
		\node(n13){$|J| = 1$ \\[2pt] \p};
	]
	[.\node(n2){$|K| = 1$ \\[2pt] \snp};
		[.\node(n21){$|I| = 2$ \\[2pt] \snp};
			\node(n211){$|J| = 2$ \\[2pt] \np};
		]
		[.\node(n23){$|J| = 2$ \\[2pt] \snp};
			\node(n231){$|I| = 2$ \\[2pt] \np};
		]
	]
	[.\node(n3){bounded in-\slash out-degrees};
		\node(n31){$|\Aout{v}| \leqslant 2,$ \\ $v \in I \cup L$ \\[2pt] \snp};
		\node(n32){$|\Ain{v}| \leqslant 2,$ \\ $v \in L \cup J$ \\[2pt] \snp};
	]
]
\begin{comment}
% Draw NP box
\draw[dotted] plot [smooth cycle,tension=0.25] coordinates {
	($(n211.north west)+(0,0.25)$)
	($(n211.south west)$)
	($(n231.south east)$)
	($(n231.north east)+(0,0.25)$)
};
% Draw P/sNP line left of NP box
\draw[dotted] plot [smooth,tension=0.5] coordinates {
	($(n11.north west)+(0,1)$)
	($(n211.north west)+(-2,0.5)$)
	($(n211.north west)+(0,0.25)$)
};
% Draw P/sNP line right of NP box
\draw[dotted] plot [smooth,tension=1] coordinates {
	($(n231.north east)+(0,0.25)$)
	($(n231.north east)+(4,0.25)$)
};
\end{comment}
% Draw "this paper" rectangle
\draw[dotted] ($(n121.north west)+(-0.1,0.1)$) rectangle ($(n121.south east)+(0.1,0)$) node[below] {\hspace{-46pt}this paper};
\end{tikzpicture}
\caption{Overview of known complexity results in a tree structure. For simplicity, we omit \#11 and \#14 from Table~\ref{tab:complexityResultsTable}.}
\label{fig:complexityResultsFigure}
\end{figure}

\section{The pooling problem with one pool and a bounded number of inputs}

\begin{comment}
\begin{figure}
\centering
\begin{tikzpicture}
\tikzset{n/.style={draw,circle,fill=white,outer sep=2pt,inner sep=0pt,minimum size=20pt}}
\tikzset{a/.style={above,midway}}
% Nodes
%  Input nodes
\node[n] (v1) at (-2.5,1) {$v_1$};
\node[n] (v2) at (-2.5,-.3) {$v_2$};
\node[draw=none] (vvdots) at (-2.5,-1.1) {$\vdots$};
\node[n] (vm) at (-2.5,-2) {$v_m$};
%  Pool node
\node[n] (p) at (0,0) {$\ell$};
%  Output nodes
\node[n] (w1) at (2.5,1) {$w_1$};
\node[n] (w2) at (2.5,-.3) {$w_2$};
\node[draw=none] (wvdots) at (2.5,-1.1) {$\vdots$};
\node[n] (wn) at (2.5,-2) {$w_n$};
% Arcs
%  Input-to-pool arcs
\draw[->] (v1) -- (p) node[a,xshift=5pt] {$a_1$};
\draw[->] (v2) -- (p) node[a,xshift=-10pt] {$a_2$};
\draw[->] (vm) -- (p) node[a,xshift=-10pt] {$a_m$};
%  Pool-to-output arcs
\draw[->] (p) -- (w1) node[a,xshift=-5pt] {$a_{m + 1}$};
\draw[->] (p) -- (w2) node[a,xshift=10pt] {$a_{m + 2}$};
\draw[->] (p) -- (wn) node[a,xshift=10pt] {$a_{m + n}$};
\end{tikzpicture}
\caption{Graph of the pooling problem with one pool and a bounded number of inputs}
\label{fig:graph}
\end{figure}
\end{comment}

In this section, we consider the pooling problem with
\begin{itemize}[label=\raisebox{0.25ex}{\tiny$\bullet$}]
\setlength\itemsep{-5pt}
\item $m$ inputs (let $I = \{ v_1, \ldots, v_m \}$),
\item one pool (let $L = \{ \ell \}$),
\item $n$ outputs (let $J = \{ w_1, \ldots, w_n \}$,
\item $q$ qualities (let $K = \{ 1, \ldots, q \}$),
\item the set of input-to-pool arcs $A_I = \{ a_1, \ldots, a_m \} = \{ (v_1, \ell), \ldots, (v_m, \ell) \}$, and
\item the set of pool-to-output arcs $A_J = \{ a_{m + 1}, \ldots, a_{m + n} \} = \{ (\ell, w_1), \ldots, (\ell, w_n) \}$.
\end{itemize}
%The graph of this special case of the pooling problem is shown in Figure~\ref{fig:graph}.
We write
\begin{itemize}[label=\raisebox{0.25ex}{\tiny$\bullet$}]
\setlength\itemsep{-5pt}
\item $x_i$ for the flow on input-to-pool arc $a_i$ $(i = 1, \ldots, m)$,
\item $y_j$ for the flow on pool-to-output arc $a_{m + j}$ $(j = 1, \ldots, n)$,
\item $c_i$ for the cost of flow on arc $a_i$ $(i = 1, \ldots, m + n)$,
\item $\lambda_{ik}$ for the $k$-th quality value at the tail node of input-to-pool arc $a_i$ $(i = 1, \ldots, m)$, and
\item $\mu_{jk}$ for the bound on the $k$-th quality value at the head node of arc $a_{m + j}$ $(j = 1, \ldots, n)$.
\end{itemize}
For a positive integer $N$, we use $[N]$ to denote the set $\{1,2,\ldots,N\}$. If for some $j \in[n]$, there exists a $k \in[q]$ such that $\min \{
\lambda_{ik}: \, i\in[m]\} > \mu_{jk}$, then $y_j = 0$ in every feasible solution. Hence,
without loss of generality, we assume
\begin{equation}\label{eq:feasibility_assumption}
\forall j \in [n] \enspace \forall k \in [q] \qquad \min \{ \lambda_{ik}: \, i \in[m] \} \leqslant \mu_{jk}.
\end{equation}
Note that $y_j > 0$ implies
\begin{equation}\label{eq:reachable_output}
\forall k\in[q]\qquad\sum \limits_{i = 1}^m \lambda_{ik} x_i \leqslant \mu_{jk} \sum \limits_{i = 1}^m x_i.  
\end{equation}
It has been observed, for instance in~\cite{Haugland15}, that for a fixed set $J'\subseteq[n]$ of
outputs, an optimal solution that satisfies the quality constraints for all $j\in J'$ and has $y_j=0$
for all $j\in[n]\setminus J'$, can be found by solving the following linear program which we denote
by $\LP(J')$:
\begin{align*}
\underset{\vect{x}, \vect{y}}{\min} \qquad  \sum \limits_{i = 1}^m c_i x_i &+ \sum \limits_{j \in J'} c_j y_j \\
\textnormal{s.t.} \qquad  (\vect{x}, \vect{y}) &\in \mathcal{F}, \\
 \sum \limits_{i = 1}^m x_i &= \sum \limits_{j \in J'} y_j, \\
 \sum \limits_{i = 1}^{m - 1} (\lambda_{ik} - \lambda_{mk}) x_i &\leqslant (\mu_{jk} - \lambda_{mk}) (x_1 + \cdots + x_m), \qquad j \in J', \ k \in [q].
\end{align*}
Let $\val(J')$ denote the optimal value of problem $\LP(J')$. An optimal solution for the
pooling problem can be obtained by solving $\LP(J')$ for every $J'\subseteq[n]$, and choosing one with
minimum $\val(J')$. Below we argue that if the number $m$ of inputs is fixed, then it is sufficient to
consider a polynomial number of subsets $J'$, where the polynomial is of degree $m-1$ in both $n$ and
$q$.

Introducing variables $z_i = x_i / \sum_{i' = 1}^m x_{i'}$ for $i
\in [m-1]$, condition~\eqref{eq:reachable_output} can be rewritten as
\begin{equation}\label{eq:reachable_output_normalized}
\forall k\in[q]\qquad\sum \limits_{i = 1}^{m - 1} (\lambda_{ik} - \lambda_{mk}) z_i \leqslant \mu_{jk} - \lambda_{mk}.  
\end{equation}

The vector $\vect{z}$ is an element of the simplex $\Delta^{m - 1} = \{ \vect{z} \in [0, 1]^{m - 1}:
\, z_1 + \cdots + z_{m - 1} \leqslant 1 \}$. For $\vect{z} \in \Delta^{m - 1}$, we define the \emph{reachable output set} $J(\vect{z})$ as
\begin{equation}\label{eq:def_J(z)}
J(\vect{z}) = \left \{ j \in[n]: \, \textnormal{\eqref{eq:reachable_output_normalized} is satisfied} \right \}.
\end{equation}

\medskip

\begin{lemma}\label{lem:lower_bound_val(z)}
The objective value for any flow corresponding to $\vect z\in\Delta^{m-1}$ is at least
$\val(J(\vect z))$.
\end{lemma}
\begin{proof}
For a fixed $\vect{z} \in \Delta^{m - 1}$, we can find the optimal flow by solving the linear program
\begin{align*}
\underset{\vect{x}, \vect{y}}{\min} \qquad  \sum \limits_{i = 1}^m c_i x_i &+ \sum \limits_{j \in J(\vect{z})} c_j y_j \\
\textnormal{s.t.} \qquad  (\vect{x}, \vect{y}) &\in \mathcal{F}, \\
 \sum \limits_{i = 1}^m x_i &= \sum_{j \in J(\vect{z})} y_j, \\
 x_i &= z_i (x_1 + \cdots + x_m), \qquad i\in[m].
\end{align*}
Every feasible solution for this problem is also feasible for $\LP(J(\vect z))$ and the claim follows.
\end{proof}
The inequalities~\eqref{eq:reachable_output_normalized} define a partition of $\mathbb{R}^{m-1}$ (and
therefore of $\Delta^{m-1}$) into regions of constant $J(\vect z)$. To be more precise, let $\mathcal{H}$ be the
hyperplane arrangement $\mathcal{H} = \{ H_{jk}: \, j \in[n], \, k \in[q] \}$, where
\[H_{jk} = \left \{ \vect{z} \in \reals^{m - 1}: \, \sum \limits_{i = 1}^{m - 1} (\lambda_{ik} -
  \lambda_{mk}) z_i = \mu_{jk} - \lambda_{mk} \right \}.\]
The system $\mathcal H$ induces a partition of $\reals^{m-1}$. Let $H_{jk}^0$ and $H_{jk}^1$ be
defined by
\begin{align*}
H_{jk}^0 & = \left \{ \vect{z} \in \reals^{m - 1}: \, \sum \limits_{i = 1}^{m - 1} (\lambda_{ik} - \lambda_{mk}) z_i \leqslant \mu_{jk} - \lambda_{mk} \right \}, \\
H_{jk}^1 & = \left \{ \vect{z} \in \reals^{m - 1}: \, \sum \limits_{i = 1}^{m - 1} (\lambda_{ik} -
  \lambda_{mk}) z_ i>  \mu_{jk} - \lambda_{mk} \right \}.
\end{align*}
If, for every vector $\vect{\varepsilon} = (\varepsilon_{jk})_{j \in[n], \, k \in[q]} \in \{ 0, 1 \}^{nq}$, we define the set
\[P(\vect{\varepsilon}) = \bigcap \limits_{j = 1}^n \bigcap \limits_{k = 1}^q
H_{jk}^{\varepsilon_{jk}},\] 
then the space $\reals^{m-1}$ is the disjoint union of the sets $P(\vect\varepsilon)$, and for every
$\vect z\in\Delta^{m-1}$ the set $J(\vect z)$ is determined by the vector $\vect\varepsilon$ with $\vect z\in P(\vect\varepsilon)$.

\medskip

\begin{lemma}\label{lem:J(e)=J(z)}
For $\vect\varepsilon\in\{0,1\}^{nq}$, let $J(\vect\varepsilon)=\{j\in[n]\ :\ \forall k\in[q]\ \varepsilon_{jk}=0\}$. Then, for all $\vect\varepsilon\in\{0,1\}^{nq}$ and
for all $\vect z\in P(\vect\varepsilon)\cap\Delta^{m-1}$, we have $J(\vect z)=J(\vect\varepsilon)$. 
\end{lemma}
\begin{proof}
Let $\vect\varepsilon\in\{0,1\}^{nq}$ and $\vect z\in
P(\vect\varepsilon)\cap\Delta^{m-1}$. Then
\[  j\in J(\vect z)\quad\stackrel{\eqref{eq:def_J(z)}}{\iff}\quad\forall k\in[q]\quad \vect z\in
  H_{jk}^0\quad\stackrel{\vect z\in P(\vect\varepsilon)}{\iff}\quad\forall k\in[q]\quad \varepsilon_{jk}=0\quad\iff\quad j\in J(\vect\varepsilon).\qedhere\]
\end{proof}
It is well known that the number of nonempty sets $P(\vect\varepsilon)$ is bounded by a polynomial
of degree $m$ in $nq$ (see for example~\cite{Buck43}). However, direct application of
  \cite{Buck43} yields the upper bound $\sum_{i = 0}^{m - 1} \tbinom{n q}{i}$, which is weaker than
  the bound in the following lemma. We derive a stronger bound than \cite{Buck43} since the $n q$
  hyperplanes are partitioned into $q$ subsets of each $n$ parallel hyperplanes.

\medskip

\begin{lemma}\label{lem:nonempty_polytopes} 
  There are at most $\displaystyle \sum \limits_{i=0}^{m-1} \binom{q}{i} n^i$ vectors
  $\vect{\varepsilon} \in \{0,\,1\}^{nq}$ such that $P(\vect{\varepsilon}) \neq \emptyset$.
\end{lemma}
\begin{proof}
We denote the claim of the lemma, parameterized by the input cardinality~$m$ and the quality
  cardinality~$q$, by $C(m, q)$, and we prove this claim by induction on $m$ and $q$. Base case
  and inductive step %of this two-dimensional induction proof of $C(m, q)$ 
are as follows:
\begin{enumerate}
%\NumTabs{4} % Declares a list of 4 equally-spaced tabs, starting at the left margin and spanning \linewidth.
\item \emph{Base case:} $\forall q,m \in\{1,2,\ldots\}: \enspace C(1, q),\ C(2,q)\textnormal{
    and } C(m,1)$
%\item \emph{Induction hypothesis:} \tab{$\forall i \in [m] \setminus \{ 1 \} \enspace \forall k \in [q] \setminus \{ 1 \}: \enspace C(i - 1, k - 1) \wedge C(i, k - 1)$}
\item \emph{Inductive step:} $\forall q\in \{2,3,\ldots\},\ \forall m\in\{3,4,\ldots\}: \enspace C(m - 1, q - 1) \wedge C(m, q - 1) \implies C(m, q)$
\end{enumerate}
For $m = 1$, note that $\reals^0=\{0\}$ contains only a single point, and since the sets
$P(\vect\varepsilon)$ are disjoint there can be at most $1=\tbinom{q}{0}n^0$ nonempty sets
$P(\vect\varepsilon)$. In fact, using assumption~\eqref{eq:feasibility_assumption}, we have $P(\vect\varepsilon)\neq\emptyset\iff\vect\varepsilon=\vect 0$. For $m=2$, the $nq$ inequalities partition $\reals^1$ into at most
$1+nq=\binom{q}{0}n^0+\binom{q}{1}n^1$ intervals. For $q = 1$ and $m\geqslant 3$, the $n$ parallel
hyperplanes $H_{11}, \ldots, H_{n1}$ partition $\reals^{m - 1}$ into at most $1 + n = \tbinom{1}{0}
n^0 + \tbinom{1}{1} n^1$ parts.
Now let $q\geqslant 2$, $m\geqslant 3$, and assume that $C(m-1,q-1)$ and $C(m,q-1)$ are true. 
From $C(m,q-1)$ it follows that the system $\{ H_{jk}: \enspace j \in [n],
\enspace k \in [q - 1] \}$ cuts $\reals^{m - 1}$ into at most
\[\sum \limits_{i = 0}^{m - 1} \binom{q - 1}{i} n^i\]
parts. For every $j\in[n]$, the hyperplane $H_{jq}$ is isomorphic to $\reals^{m-2}$, and for every
$j'\in[n]$, $k\in[q-1]$, the intersection $H_{j'k}\cap H_{jq}$ is either empty or an
$(m-3)$-dimensional affine subspace of $H_{jk}$. Since the map $H_{j'k}\mapsto H_{j'k}\cap H_{jq}$ preserves
parallelism, $C(m-1,q-1)$ implies that the hyperplane $H_{jq}$
is cut by the system $\{ H_{j'k} \cap H_{jq}: \enspace j' \in [n], \enspace k \in [q - 1] \}$ into
at most
\[\sum \limits_{i = 0}^{m - 2} \binom{q - 1}{i} n^i\]
parts. If we start with the partition of $\reals^{m-1}$ given by the system $\{ H_{jk}: \enspace j \in [n],
\enspace k \in [q - 1] \}$ and add the hyperplanes $H_{1q}$, $H_{2q}$,\ldots, $H_{nq}$ one by one,
then every hyperplane adds at most $\sum_{i = 0}^{m - 2} \binom{q - 1}{i} n^i$ parts to the
partition, and the number of parts into which $\reals^{m - 1}$ is cut by $\mathcal H$ is
at most
\begin{multline*}
\sum \limits_{i = 0}^{m - 1} \binom{q - 1}{i} n^i+  n \sum \limits_{i = 0}^{m - 2} \binom{q - 1}{i}
n^i = \sum \limits_{i = 0}^{m - 1} \binom{q - 1}{i} n^i + \sum \limits_{i = 1}^{m - 1} \binom{q - 1}{i - 1} n^i \\
  = \binom{q - 1}{0} n^0 + \sum_{i = 1}^{m - 1} \left ( \binom{q - 1}{i} + \binom{q - 1}{i - 1}
  \right ) n^i = \sum_{i = 0}^{m - 1} \binom{q}{i} n^i. \tag*{\qedhere}
\end{multline*}
\end{proof}
\begin{remark}
Note that the proof of Lemma~\ref{lem:nonempty_polytopes} also provides a recursive method to determine the
vectors $\vect\varepsilon$ with $P(\vect\varepsilon)\neq\emptyset$ in polynomial time.  
\end{remark}
\begin{remark}
  The upper bound given in Lemma~\ref{lem:nonempty_polytopes} is best possible, i.e., for all $m$,
  $q$ and $n$, there exist instances in which the number of vectors $\vect\varepsilon$ with
  $P(\vect\varepsilon) \neq \emptyset$ equals $\sum_{i=0}^{m-1} \binom{q}{i} n^i$. In fact, this bound is obtained by almost all systems
  $\mathcal H$. To make this statement more precise, we say that a system $\mathcal H$ of $nq$ hyperplanes $H_{jk}$
  in $\reals^{m-1}$, consisting of $q$ sets of $n$ parallel hyperplanes, is in \emph{general position} if
  the intersection of every set of $m$ of these hyperplanes is empty and
  \begin{multline*}
\forall t\in[m-1]\quad \forall
(j_1,\ldots,j_t)\in[n]^t\quad \forall (k_1,\ldots,k_t)\in[q]^t\text{ with }k_1<k_2<\cdots< k_t \\
H_{j_1k_1}\cap H_{j_2k_2}\cap\cdots\cap H_{j_tk_t}\text{ is an $(m-1-t)$-dimensional affine subspace
  of $\reals^{m-1}$.}    
  \end{multline*}
  The bound in Lemma~\ref{lem:nonempty_polytopes} is obtained whenever the system
  $\mathcal H$ is in general position, and this can be seen by checking that in this case all
  estimates in the induction proof are tight. For $m=1$, we have that $P(\vect
  0)=\{0\}\neq\emptyset$. For $m=2$, the system $\mathcal H$ is a list of $nq$ points, and $\mathcal
  H$ is in general position if these points are distinct, in which case it partitions $\reals^{m-1}$
  into $1+nq$ parts as required. For $q=1$, the $n$ parallel hyperplanes $H_{11}, \ldots, H_{n1}$ in
  general position partition $\reals^{m - 1}$ into exactly $1 + n$ parts. For the inductive step,
  note that the system of intersections $\{H_{j'k}\cap H_{jq}\ :\ j'\in[n],\ k\in[q-1]\}$ forms a
  system of hyperplanes in general position in $H_{jq}$, and therefore the inequalities in the
  inductive step are satisfied with equality.
\end{remark}

\medskip

\begin{theorem}\label{thm:polynomial_time}
  For every positive integer $m$, the pooling problem with one pool and $m$ inputs can be solved in
  polynomial time. More precisely, it can be reduced to solving at most
\[\sum \limits_{i = 0}^{m-1} \binom{q}{i} n^i\]
linear programs with $m+n$ variables and $m+n(q+1)+2$ constraints, where $q$ is the number of qualities and $n$ is the number of outputs.
\end{theorem}
\begin{proof}
  We claim that the pooling problem can be solved by choosing a minimum cost solution obtained from
  solving the problem $\LP(J(\vect\varepsilon))$ for every $\vect\varepsilon$ with
  $P(\varepsilon)\cap\Delta^{m-1}\neq\emptyset$, and by Lemma~\ref{lem:nonempty_polytopes} the
  number of these linear programs is bounded as claimed. Clearly,
  $B=\min\{\val(J(\vect\varepsilon))\ :\ P(\vect\varepsilon)\cap\Delta^{m-1}\neq\emptyset\}$ is an
  upper bound because a solution for $\LP(J(\vect\varepsilon))$ is always feasible for the pooling
  problem. By Lemma~\ref{lem:J(e)=J(z)}, for every $\vect z\in\Delta^{m-1}$ there exists some
  $\vect\varepsilon$ with $J(\vect z)=J(\vect\varepsilon)$, and using
  Lemma~\ref{lem:lower_bound_val(z)} it follows that $B$ is also a lower bound.
\end{proof}

\annotate{We note that this result was obtained, independently, by Haugland and Hendrix \cite{Haugland16}.}

\section{Remaining open problems}

To further characterize the complexity of the pooling problem, the following open problems could be addressed in the future \cite{Haugland14,Haugland15}:
\begin{enumerate}
\setlength\itemsep{-5pt}
\item For all the cases that can be solved in polynomial time by reduction to polynomially many
  linear programs of polynomial size, does there exist a \emph{strongly} polynomial algorithm, i.e.,
  an algorithm that is polynomial in the number of vertices and qualities?
\item Is the pooling problem with one quality and in-degrees at most two polynomially solvable?
\item Is the pooling problem with one quality and out-degrees at most two polynomially solvable?
\item Do polynomial algorithms exist for the pooling problem with two pools and some bounds on the
  number of inputs, outputs, and qualities?
\end{enumerate}

\vfill

\paragraph{Acknowledgements} This research was supported by the ARC Linkage Grant no. LP110200524, Hunter Valley Coal Chain Coordinator (\href{http://www.hvccc.com.au}{hvccc.com.au}) and Triple Point Technology (\href{http://www.tpt.com}{tpt.com}). We would like to thank the two anonymous referees for their helpful comments which improved the quality of the paper.

%\bibliography{References}
%\bibliographystyle{abbrv}

\end{document}